\documentclass[12pt,a4paper]{article}%
\usepackage[utf8]{inputenc}
\usepackage{hyperref}
\usepackage{amsmath}
\usepackage{amsfonts}
\usepackage{amssymb}
\usepackage{xcolor}
\usepackage{graphicx}%

\usepackage[normalem]{ulem}

\providecommand{\U}[1]{\protect\rule{.1in}{.1in}}
\newtheorem{theorem}{Theorem}

\newtheorem{conjecture}[theorem]{Conjecture}
\newtheorem{corollary}[theorem]{Corollary}

\newtheorem{lemma}[theorem]{Lemma}

\newtheorem{proposition}[theorem]{Proposition}
\newtheorem{remark}[theorem]{Remark}

\newenvironment{proof}[1][Proof]{\noindent\textbf{#1.} }{\ \hfill \rule{0.5em}{0.5em}\bigskip}
\graphicspath{{Slike/}}

\begin{document}

\title{Disproving two conjectures on the cospectrality of complete tripartite graphs}
\author{Irena M. Jovanovi\'{c}$^{1}$, Jelena Sedlar$^{2,4}$, Riste
\v{S}krekovski$^{3,4,5}$\\[0.3cm]
{\small $^{1}$ \textit{Independent Researcher, Belgrade, Serbia}}\\[0.1cm]
{\small $^{2}$ \textit{University of Split, Faculty of civil engineering,
architecture and geodesy, Croatia}}\\[0.1cm]
{\small $^{3}$ \textit{Faculty of Mathematics and Physics, University of
Ljubljana, Slovenia}}\\[0.1cm]
{\small $^{4}$ \textit{Faculty of Information Studies, University of Novo Mesto, Slovenia}}\\[0.1cm]
{\small $^{5}$ \textit{FAMNIT, University of Primorska, Koper, Slovenia}}}
\date{}
\maketitle

\begin{abstract}
The spectral distance between two non-isomorphic graphs of the same order is the $\ell^p$-distance between their ordered adjacency spectra, and the cospectrality of a graph $G$ on a family $\mathcal{G}_n$ is the minimum such distance from $G$ to another graph in $\mathcal{G}_n$. Abdollahi and Zakeri (Ars Math. Contemp. 22(1) (2022) P1.06) conjectured that, for every $n\ge 4$ and for $p\in\{1,2\}$, the cospectrality of the complete tripartite graph $K_{n,n,n}$ among all graphs on $3n$ vertices is attained exactly at $K_{n-1,n,n+1}$. We disprove both conjectures for every $n\ge 4$. As a counterexample we use the graph $\mathcal{H}_n=(n-1)K_1\nabla(K_1\cup K_{n,n})$, whose spectrum we determine. It yields $\mathrm{cs}^{(1)}(K_{n,n,n})\le 1+o(1)$ and $\mathrm{cs}^{(2)}(K_{n,n,n})\le \tfrac12+o(1)$ as $n\to\infty$, contrary to the conjectured values. We also study the cospectrality of $K_{n,n,n}$ for $p=2$, asymptotically, among complete multipartite graphs with possibly added isolated vertices.
\end{abstract}



\section{Introduction}

All graphs considered in this paper are finite, simple and undirected. For such a graph $G$ of order $n$ with the vertex set $V(G)=\{v_1,v_2,\ldots,v_n\}$, the matrix $A(G)=[a_{ij}]_{n\times n}$, whose $(i,j)$-th entry $a_{ij}$ is equal to the number of edges connecting vertices $v_i$ and $v_j$, is the \emph{adjacency matrix} of $G$. The \emph{cha\-ra\-cte\-ri\-stic polynomial} $P_G(x)=\det(xI_n-A(G))$ of $G$, where $I_n$ stands for the identity matrix, is the characteristic polynomial of its adjacency matrix $A(G)$. The adjacency eigenvalues $\lambda_1(G)\geq\lambda_2(G)\geq\cdots\geq\lambda_n(G)$ of $G$ are the eigenvalues of its adjacency matrix, and they constitute the \emph{(adjacency) spectrum} of $G$. We will always assume that the eigenvalues of graphs considered throughout the paper are ordered monotonically decreasing. For an integer $l\geq 1$, the \emph{$l$-th spectral moment} $s_l(G)$ of $G$ is defined as $s_l(G)=\sum_{i=1}^{n} \lambda_i^l(G)$. It holds \cite{CvRowSim} that $s_1(G)=0$ and $s_2(G)=2m$, where $m$ is the number of edges of $G$. Let us denote $\Gamma_{G}=(\lambda_1(G), \lambda_2(G),\ldots,\lambda_n(G))$. 

The notation common for the spectral graph theory is used in the paper. In that way, $K_n$ is the complete graph on $n$ vertices, while for the complete multipartite graph with $r$ parts and $n_i$, $1\leq i\leq r$, vertices in each of them, the mark $K_{n_1,n_2,\ldots,n_r}$ will be used. The characteristic polynomials and spectra of these two graphs are well known, and can be found, for example, in \cite[Section 2.6]{CvRowSim}. That way, the spectrum of $K_n$ consists of the eigenvalue $n-1$ and $n-1$ eigenvalues equal to $-1$, while the spectrum of $K_{n,n,n}$ contains the following eigenvalues $2n$, $0$ with the multiplicity $3n-3$ and $-n$ with the multiplicity $2$. The \emph{complement} of a graph $G$ is denoted by $\overline{G}$. For the remaining notation and terminology, the reader is referred to \cite{CvDoobSachs} and \cite{CvRowSim}.

Let $\mathbf{x}=(x_1,x_2,\ldots,x_n)^T$ be a vector in the $n$-dimensional real vector space $\mathbb{R}^n$. For a real number $p\geq 1$, the \emph{$\ell^p$-norm} of $\mathbf{x}$ is
\begin{equation*}
\|\mathbf{x}\|_p=\left(\sum\limits_{i=1}^{n} |x_i|^p\right)^{\frac{1}{p}}.
\end{equation*}

Let $G_1$ and $G_2$ be two non-isomorphic graphs on $n$ vertices. The \emph{spectral distance $\sigma^{(p)}(G_1, G_2)$} between $G_1$ and $G_2$ is defined as
\begin{equation*}
\sigma^{(p)}(G_1, G_2)=\left(\|\Gamma_{G_1}-\Gamma_{G_2}\|_p\right)^p=\sum_{i=1}^{n}\left\vert\lambda_{i}(G_1)-\lambda_{i}(G_2)\right\vert^p .
\end{equation*}

\noindent If $p=1$, the distance $\sigma^{(1)}(G_1, G_2)=\sum_{i=1}^{n}\left\vert\lambda_{i}(G_1)-\lambda_{i}(G_2)\right\vert$ is known as the \emph{Manhattan spectral distance}, for $p=2$, $\sigma^{(2)}(G_1, G_2)=\sum_{i=1}^{n}(\lambda_i(G_1)-\lambda_i(G_2))^2$ is the \emph{Euclidean spectral distance}. In the case of the Manhattan spectral distance, the superscript in the notation will be omitted.

Let $\mathcal{G}$ be a set of non-isomorphic graphs on $n$ vertices. The \emph{cospectrality} ${\rm cs}^{(p)}_{\mathcal{G}}(G)$ of a graph $G\in \mathcal{G}$, introduced in \cite[Problem AWGS.4]{aveiro} for $p\in\{1,2\}$, is 
\begin{equation*}
{\rm cs}^{(p)}_{\mathcal{G}}(G)=\min\{\sigma^{(p)}(G, H): H\in\mathcal{G}, H\ncong G\}.
\end{equation*}
The superscript is omitted when $p=1$.

In the literature so far, spectral distances between graphs and the related parameters have been studied mostly for $p=1$ and $p=2$. Jovanovi\'c and Stani\'c \cite{uma22,uma15,JovStaFIL,JovStaLAA} considered them with respect to the $\ell^1$-norm on particularly constructed classes of graphs, while Abdollahi and Zakeri \cite{az-cospectrality} determined the cospectralities of $nK_1$, $K_2+(n-2)K_1$, $K_n$, $K_{n,n}$ and $K_{n,n+1}$; the cospectrality with respect to the $\ell^2$-norm was studied in \cite{janbaz,ao-cospectrality,oboudi} for several complete bipartite graphs and for $nK_1$, $K_2+(n-2)K_1$ and $K_n$. In \cite{az-arscontemp}, Abdollahi and Zakeri turned to the cocktail party graph $CP_n$ and the complete tripartite graph $K_{n,n,n}$, and, based on computational results, posed conjectures determining their cospectralities for $p\in\{1,2\}$. Here we consider the conjectures concerning $K_{n,n,n}$.

\begin{conjecture}\cite[Conjecture 3.9]{az-arscontemp}\label{conjecture:triparts:1}
Let $\mathcal{G}$ be the set of non-isomorphic graphs on $3n$ vertices, $n\geq 4$. Then, ${\rm cs}_{\mathcal{G}}(K_{n,n,n})=\sigma(K_{n,n,n},H)$, for some graph $H$ if and only if $H\cong K_{n-1,n,n+1}$.
\end{conjecture}

\begin{conjecture}\cite[Conjecture 3.10]{az-arscontemp}\label{conjecture:triparts:2}
Let $\mathcal{G}$ be the set of non-isomorphic graphs on $3n$ vertices, $n\geq 4$. Then, ${\rm cs}_{\mathcal{G}}^{(2)}(K_{n,n,n})=\sigma^{(2)}(K_{n,n,n},H)$, for some graph $H$ if and only if $H\cong K_{n-1,n,n+1}$. 
\end{conjecture}

Notice that both conjectures are stated for $n\geq4$. Namely, in \cite{az-arscontemp} it is established that for $n=3$ the graph nearest to $K_{3,3,3},$ with respect to both the Manhattan and the Euclidean spectral distance and uniquely so, is not $K_{2,3,4}$ but a certain graph of order nine which is presented there in Figure 1. The case $n=3$ was therefore considered to be exceptional.

In this paper, we introduce a graph $\mathcal{H}_{n}$ of order $3n$, and prove that $\mathcal{H}_{n}$ is closer to $K_{n,n,n}$ than $K_{n-1,n,n+1}$ is, in both distances and for every $n\geq3$. The graph $\mathcal{H}_{3}$ is the unique graph nearest to $K_{3,3,3}$, so it is the graph meant in \cite{az-arscontemp} (the graph drawn in Figure 1 in \cite{az-arscontemp} differs from $\mathcal{H}_3$ by two non-adjacent edges). Hence both conjectures above are false already for $n=4.$ Moreover, we determine the asymptotic behaviour of the Manhattan and the Euclidean spectral distance between $K_{n,n,n}$ and $K_{n-1,n,n+1}$, and between $K_{n,n,n}$ and $\mathcal{H}_n$, which shows that the conjectured values of the cospectrality are wrong by a constant factor. In addition, the cospectrality of $K_{n,n,n}$ with respect to the Euclidean distance is considered asymptotically on the set of all complete multipartite graphs, possibly with isolated vertices added.

\section{Graph $\mathcal{H}_n$ and its properties}

In this section, we introduce the graph which served to refute Conjectures \ref{conjecture:triparts:1} and \ref{conjecture:triparts:2}, and expose some of its basic properties.

First of all, let us recall that the \emph{join} $G_1\nabla G_2$ of (disjoint) graphs $G_1$ and $G_2$ is the graph obtained from their disjoint union $G_1\,\cup\,G_2$ by joining each vertex of $G_1$ to each vertex of $G_2$. The characteristic polynomial of the join can be found in \cite[Theorem 2.1.5]{CvRowSim}.

\begin{theorem}\label{join}
Let $G_1$ and $G_2$ be graphs with $n_1$ and $n_2$ vertices, respectively. The characteristic polynomial of the join $G_1\nabla G_2$ is given by the following relation
\begin{align*}
P_{G_1\nabla G_2}(x) =\, & (-1)^{n_2}P_{G_1}(x)P_{\overline{G}_2}(-x-1)+(-1)^{n_1}P_{G_2}(x)P_{\overline{G}_1}(-x-1)\\
                   & -(-1)^{n_1+n_2}P_{\overline{G}_1}(-x-1)P_{\overline{G}_2}(-x-1).
\end{align*}
\end{theorem}

For $n\geq2$, let us denote $\mathcal{H}_{n}=(n-1)K_{1}\nabla\left( K_{1}\cup K_{n,n}\right)$. The graph $\mathcal{H}_{n}$ is of order $3n$, it has $3n^{2}-n-1$ edges and it can be obtained by adding a new vertex $v$ to the complete tripartite graph $K_{n-1,n,n}$ with parts $X$, $Y$ and $Z$, where $|X|=n-1$ and $|Y|=|Z|=n$, and joining $v$ to every vertex of $X$ and to no other vertex. The spectrum of $\mathcal{H}_{n}$ is given by the following theorem.

\begin{theorem}\label{spec_hn}
For every $n\geq2,$ the spectrum of $\mathcal{H}_{n}$ consists of $-n$, $0$ with the multiplicity $3n-4$, and three eigenvalues $\lambda_1(\mathcal{H}_{n}), \lambda_2(\mathcal{H}_{n})$ and $\lambda_{3n-1}(\mathcal{H}_{n})$ which are the roots of the polynomial
\begin{equation}
q(x)=x^{3}-nx^{2}-(2n^{2}-n-1)x+(n^{2}-n).\label{poly_q}
\end{equation}
Moreover, $\lambda_{1}(\mathcal{H}_{n})\in\left(\frac{1}{2},2n\right),$ $\lambda_{2}(\mathcal{H}_{n})\in\left(0,\frac{1}{2}\right)$
and $\lambda_{3n-1}(\mathcal{H}_{n})\in(-n,0).$
\end{theorem}

\begin{proof}
Let $H_1=(n-1)K_1$ and $H_2=K_1\cup K_{n,n}$. Using \cite[Theorem 2.1.1]{CvRowSim}, the characteristic polynomials of $H_1$ and $H_2$ are easily found
\begin{equation}\label{char_h1}
P_{H_1}(x)=x^{n-1},
\end{equation}
\begin{equation}\label{char_h2}
P_{H_2}(x)=(x-n)(x+n)\,x^{2n-1}.
\end{equation}
Since $\overline{H}_1=K_{n-1}$ and $\overline{H}_2=K_1\nabla 2K_n$, we have
\begin{equation}\label{char_comp_h1}
P_{\overline{H}_1}(x)=(x-n+2)(x+1)^{n-2},
\end{equation}
while using the fact $\overline{2K_n}=K_{n,n}$ and Theorem \ref{join},
\begin{equation}\label{char_comp_h2}
P_{\overline{H}_2}(x)=(x+1-n)(x+1)^{2n-2}(x^2+(1-n)x-2n).
\end{equation}
Given that $\mathcal{H}_{n}=H_1\nabla H_2$, by calling Theorem \ref{join} and using equations (\ref{char_h1}) -- (\ref{char_comp_h2}), we obtain
\begin{align*}
P_{\mathcal{H}_{n}}(x) = & \, (x+n)\,x^{3n-3}((x+1)^2+(n-1)(x+1)-2n)\\
                         & \, +(x+n)(x-n)(x+n-1)\,x^{3n-3}\\
                         & \, -(x+n)(x+n-1)\,x^{3n-4}\,((x+1)^2+(n-1)(x+1)-2n),
\end{align*}
that is
\begin{align}
\nonumber P_{\mathcal{H}_{n}}(x) = & \, x^{3n-4}\,(x+n)\,(x^3-nx^2+(1+n-2n^2)\,x+n^2-n)\\
\label{char_hn}                       = & \, x^{3n-4}\,(x+n)\,q(x).
\end{align}
Therefore, the first part of the statement follows from (\ref{char_hn}).

It remains to establish the position of the roots of the polynomial $q(x)$, i.e. the position of the eigenvalues in the spectrum of $\mathcal{H}_{n}$. We have $q(0)=n^{2}-n>0$,  $q\left(\frac{1}{2}\right)=-\frac{3}{4}\left(n-\frac{5}{6}\right)<0$, $q(2n)=3n\left( n+\frac{1}{3}\right) >0$ and $q(-n)=-2n<0$, so the
three roots are separated as claimed. This completes the proof.
\end{proof}

Let us notice that the bounds of intervals to which the eigenvalues $\lambda_{1}(\mathcal{H}_{n})$ and $\lambda_{2}(\mathcal{H}_{n})$ of $\mathcal{H}_{n}$ belong can be improved. Namely, from the proof of Theorem \ref{spec_hn}, we have that $q(x)$ is a monic polynomial of degree three, which is positive on the intervals $(\lambda_{3n-1}(\mathcal{H}_{n}), \lambda_{2}(\mathcal{H}_{n}))$ and $(\lambda_{1}(\mathcal{H}_{n}),+\infty)$, and negative on the intervals $(-\infty,\lambda_{3n-1}(\mathcal{H}_{n}))$ and $(\lambda_{2}(\mathcal{H}_{n}),\lambda_{1}(\mathcal{H}_{n}))$. For $n\geq 3$, it holds
\begin{equation}\label{q_fist}
q\left(2n-\frac{1}{2}-\frac{1}{3n}\right)=-\frac{1}{216n^3}\cdot\left(54n^4-153n^3+6n^2+36n+8\right)<0,
\end{equation}
while for $n\geq 2$, 
\begin{equation}\label{q_second}
q\left(\frac{1}{2}-\frac{1}{2n}\right)=\frac{1}{8n^3}\cdot\left(2n^4+5n^3-9n^2+3n-1\right)>0.
\end{equation}
From the proof of Theorem \ref{spec_hn}, we have $q\left(\frac{1}{2}\right)<0$ and $q(2n)>0$, which, together with (\ref{q_fist}) and (\ref{q_second}), proves the following statement. 

\begin{lemma}\label{spec_bounds_hn}
For every $n\geq 3$, the three roots of the polynomial (\ref{poly_q}) satisfy the following $\lambda_{1}(\mathcal{H}_{n})\in\left(2n-\frac{1}{2}-\frac{1}{3n},2n\right)$, $\lambda_{2}(\mathcal{H}_{n})\in\left(\frac{1}{2}-\frac{1}{2n},\frac{1}{2}\right)$ and $\lambda_{3n-1}(\mathcal{H}_{n})\in(-n,0).$
\end{lemma}

In the following, closed form expressions for the Manhattan and the Euclidean spectral distances of $\mathcal{H}_{n}$ and $K_{n,n,n}$, which will be used in the next section, are exposed.

\begin{theorem}\label{hn_distances}
For $n\geq 2$, 
$$\sigma(K_{n,n,n},\mathcal{H}_{n})=2\Lambda\text{ \ and \ }\sigma^{(2)}(K_{n,n,n},\mathcal{H}_{n})=2\left( \Lambda^{2}-\Lambda\lambda_{2}(\mathcal{H}_{n})+\lambda_{2}^{2}(\mathcal{H}_{n})\right),$$
where $\Lambda=2n-\lambda_1(\mathcal{H}_{n})$, while $\lambda_1(\mathcal{H}_{n})$ and $\lambda_2(\mathcal{H}_{n})$ are as in Theorem \ref{spec_hn}.
\end{theorem}

\begin{proof}
As we have noticed in the introductory part, the spectrum of $K_{n,n,n}$ consists of the following eigenvalues $2n$, $0$ of the multiplicity $3n-3$, and $-n$ of the multiplicity $2$. Hence, having in mind Theorem \ref{spec_hn}, we easily find
\begin{align*}
\sigma(K_{n,n,n},\mathcal{H}_{n}) = & \,|2n-\lambda_1(\mathcal{H}_{n})|+|\lambda_2(\mathcal{H}_{n})|+|n+\lambda_{3n-1}(\mathcal{H}_{n})|\\
                                  = & \,3n-\lambda_1(\mathcal{H}_{n})+\lambda_2(\mathcal{H}_{n})+\lambda_{3n-1}(\mathcal{H}_{n})=4n-2\lambda_1(\mathcal{H}_{n})=2\Lambda,
\end{align*}
since the sum of the adjacency eigenvalues of a graph is equal to zero, i.e. $\lambda_1(\mathcal{H}_{n})+\lambda_2(\mathcal{H}_{n})+\lambda_{3n-1}(\mathcal{H}_{n})=n$.

In a similar way, we have 
\begin{align*}
\sigma^{(2)}(K_{n,n,n},\mathcal{H}_{n})&=(2n-\lambda_1(\mathcal{H}_{n}))^2+\lambda_2^2(\mathcal{H}_{n})+(n+\lambda_{3n-1}(\mathcal{H}_{n}))^2\\
&=\Lambda^2+\lambda_2^2(\mathcal{H}_{n})+(\Lambda-\lambda_2(\mathcal{H}_{n}))^2,
\end{align*}
wherefrom the proof follows.
\end{proof}

\begin{remark}\label{remark}
For $n\geq 3$, from Lemma \ref{spec_bounds_hn} it follows $0<\Lambda<\frac{1}{2}+\frac{1}{3n}$.
\end{remark}

\section{Counterexamples to the conjectures}

In this section, we present observations and proofs which disprove Conjectures \ref{conjecture:triparts:1} and \ref{conjecture:triparts:2}. In support of that, the following results, which have been proved in \cite[Lemma 2]{Harary} and in \cite[Corollary 1]{Harary}, respectively, will be of help.

\begin{lemma}\label{comp_multi_part}
For a complete $r$-partite graph $K_{p_1,p_2,\ldots,p_r}$, where $\sum_{i=1}^{r}p_i=p$ and $p_1\leq p_2\leq\cdots\leq p_r$, the characteristic polynomial $P_{K_{p_1,p_2,\ldots,p_r}}(x)$ is given by 
\begin{equation*}
P_{K_{p_1,p_2,\ldots,p_r}}(x)=x^{p-r}\,\left[\prod\limits_{i=1}^{r} (x+p_i)-\sum\limits_{i=1}^{r} p_i\prod\limits_{j=1,j\neq i}^{r} (x+p_j)\right].
\end{equation*}
\end{lemma}

\begin{corollary}\label{cmp_largest_eigen}
For a complete $r$-partite graph $K_{p_1,p_2,\ldots,p_r}$, where $\sum_{i=1}^{r}p_i=p$ and $p_1\leq p_2\leq\cdots\leq p_r$, the largest eigenvalue satisfies
\begin{equation*}
p-p_r\leq\lambda_1(K_{p_1,p_2,\ldots,p_r})\leq p-p_1.
\end{equation*}
\end{corollary}

According to Lemma \ref{comp_multi_part}, the characteristic polynomial of the complete tripartite graph $K_{a,b,c}$, where $a,b,c\geq 1$, is
\begin{align*}
P_{K_{a,b,c}}(x) = &\, x^{a+b+c-3}\cdot(x^3-(ab+ac+bc)x-2abc)\\
                 = &\, x^{a+b+c-3}\cdot(x^3-\alpha x-\beta)=x^{a+b+c-3}\cdot\phi(x),
\end{align*}
where $\phi(x)=x^3-\alpha x-\beta$, $\alpha=ab+ac+bc$ and $\beta=2abc$. This means that the spectrum of $K_{a,b,c}$ consists of $0$ with the multiplicity $a+b+c-3$, and three eigenvalues which are the roots of the polynomial $\phi(x)$. 

The next statement can be found in \cite[Theorem 6.7]{CvDoobSachs}.

\begin{proposition}\label{lambda2=0}
A graph has exactly one positive eigenvalue if and only if its non-isolated vertices form a complete multipartite graph.
\end{proposition} 

Accordingly, based on Proposition \ref{lambda2=0}, exactly one root of the polynomial $\phi(x)$, i.e. one of the eigenvalues of $K_{a,b,c}$ is positive. Since $\beta>0$, by applying Vieta's formulas, one can easily verify that the remaining two roots of $\phi(x)$ are negative reals. 

\begin{lemma}\label{helpful}
Let $a,b,c\geq1$, where $a+b+c\leq3n$, and let $t=3n-a-b-c$. Then
$$\sigma^{(2)}(K_{n,n,n},tK_{1}\cup K_{a,b,c})=6n^{2}+2\alpha-6n\lambda_1(K_{a,b,c}),$$
where $\alpha=ab+ac+bc$.
\end{lemma}

\begin{proof}
The eigenvalues of $K_{n,n,n}$ are $2n$, $0$ with the multiplicity $3n-3$ and $-n$ with the multiplicity $2$, while the spectrum of $tK_{1}\cup K_{a,b,c}$ consists of $\lambda_1(tK_{1}\cup K_{a,b,c})$, $0$ with the multiplicity $3n-3$, $\lambda_{3n-1}(tK_{1}\cup K_{a,b,c})$ and $\lambda_{3n}(tK_{1}\cup K_{a,b,c})$. Let us note that the eigenvalues $\lambda_1(tK_{1}\cup K_{a,b,c})$, $\lambda_{3n-1}(tK_{1}\cup K_{a,b,c})$ and $\lambda_{3n}(tK_{1}\cup K_{a,b,c})$ are actually non-zero eigenvalues of $K_{a,b,c}$. Hence, we easily find
\begin{align*}
\sigma^{(2)}(K_{n,n,n},tK_{1}\cup K_{a,b,c}) = &\, (2n-\lambda_1(tK_{1}\cup K_{a,b,c}))^2+(n+\lambda_{3n-1}(tK_{1}\cup K_{a,b,c}))^2\\
                                               &\, +(n+\lambda_{3n}(tK_{1}\cup K_{a,b,c}))^2\\
                                             = &\, 6n^2-4n\lambda_{1}(tK_{1}\cup K_{a,b,c})+2n\,\lambda_{3n-1}(tK_{1}\cup K_{a,b,c})\\
                                               &\, +2n\,\lambda_{3n}(tK_{1}\cup K_{a,b,c})+\lambda_{1}^2(tK_{1}\cup K_{a,b,c})\\
                                               &\, +\lambda_{3n-1}^2(tK_{1}\cup K_{a,b,c})+\lambda_{3n}^2(tK_{1}\cup K_{a,b,c})\\
                                             = &\, 6n^{2}+2\alpha-6n\lambda_1(tK_{1}\cup K_{a,b,c}),
\end{align*} 
where it is used that the first and the second spectral moment of $K_{a,b,c}$, that is of $tK_{1}\cup K_{a,b,c}$ are $s_1(tK_{1}\cup K_{a,b,c})=0$ and $s_2(tK_{1}\cup K_{a,b,c})=2\alpha$, respectively.
\end{proof}

Let us now consider the special case $K_{n-1,n,n+1}$ of the complete tripartite graph $K_{a,b,c}$. Here, $\alpha=3n^2-1$ and $\beta=2\,(n^3-n)$. According to Corollary \ref{cmp_largest_eigen} and Proposition \ref{lambda2=0}, the largest eigenvalue $\lambda_1(K_{n-1,n,n+1})$ of this graph is its only positive eigenvalue which belongs to the interval $[2n-1,2n+1]$. But, since $\phi(0)=-2n(n^2-1)<0$ and $\phi(2n)=4n>0$, for $n\geq2$, we have $\lambda_1(K_{n-1,n,n+1})\in[2n-1,2n)$. Therefore, $\lambda_1(K_{n-1,n,n+1})=2n+\delta$, where $\delta\in[-1,0)$, and 
\begin{equation*}
0=\phi(\lambda_1(K_{n-1,n,n+1}))=\phi(2n+\delta)=4n+\delta^3+6n\delta^2+(9n^2+1)\delta.
\end{equation*}   
Let us denote $\Phi(\delta)=4n+\delta^3+6n\delta^2+(9n^2+1)\delta$, and let us notice that $\Phi^{\prime}(\delta)=3\delta^2+12n\delta+9n^2+1>0$, for $\delta\in[-1,0)$ and $n\geq 2$, as well as $\Phi(0)=4n>0$ and $\Phi(-1)=-9n^2+10n-2<0$. This means that $\Phi(\delta)$ has the unique root $\delta=\delta_n$ in the interval $(-1,0)$, where $\Phi(\delta)$ is increasing. It can be easily verified that $\Phi\left(-\frac{4}{9n}\right)=\frac{4}{729n^3}(135n^2-16)>0$ and $\Phi\left(-\frac{1}{n}\right)=-\frac{5n^4-5n^2+1}{n^3}<0$, for $n\geq2$, which gives
\begin{equation}\label{root_delta}
\delta_n\in\left(-\frac{1}{n},-\frac{4}{9n}\right).
\end{equation}  

Now, we are ready to compute and estimate the spectral distances between complete tripartite graphs $K_{n,n,n}$ and $K_{n-1,n,n+1}$.

\begin{lemma}\label{euclid_dist}
For every $n\geq 2$, 
$$\sigma^{(2)}(K_{n,n,n}, K_{n-1,n,n+1})=6n(2n-\lambda_1(K_{n-1,n,n+1}))-2>\frac{2}{3}.$$
\end{lemma}

\begin{proof}
By applying Lemma \ref{helpful} for $(a,b,c)=(n-1,n,n+1)$ and having in mind previous observations together with the estimation (\ref{root_delta}), we find 
$$\sigma^{(2)}(K_{n,n,n}, K_{n-1,n,n+1})=6n(2n-\lambda_1(K_{n-1,n,n+1}))-2=-6n\delta_n-2>\frac{2}{3}.$$
\end{proof}

In proving the subsequent lemma we use the \emph{Interlacing Theorem for Complete Multipartite Graphs} \cite[Theorem 1]{Harary}.

\begin{theorem}\label{cmp_intelace}
The $r-1$ negative eigenvalues $\lambda_{p-r+2},\ldots,\lambda_p$ of a complete $r$-partite graph $K_{p_1,p_2,\ldots,p_r}$, where $\sum_{i=1}^{r}p_i=p$ and $p_1\leq p_2\leq\cdots\leq p_r$, satisfy the inequalities
$$p_1\leq-\lambda_{p-r+2}\leq p_2\leq -\lambda_{p-r+3}\leq p_3\leq\cdots\leq p_{r-1}\leq -\lambda_p\leq p_r.$$ 
\end{theorem}

\begin{lemma}\label{manhat_dist}
For every $n\geq 4$, 
\begin{align*}
\sigma(K_{n,n,n}, K_{n-1,n,n+1})&=2n-\lambda_1(K_{n-1,n,n+1})\\
&\quad+\sqrt{-3\lambda_1^2(K_{n-1,n,n+1})+12n^2-4}>\frac{2\sqrt{3}}{3}.
\end{align*}
\end{lemma}

\begin{proof}
As we have already indicated, the non-zero eigenvalues $\lambda_1(K_{n-1,n,n+1})$, then $\lambda_{3n-1}(K_{n-1,n,n+1})$ and $\lambda_{3n}(K_{n-1,n,n+1})$ of the graph $K_{n-1,n,n+1}$ are the roots of the polynomial $\phi(x)=x^3-\alpha x-\beta=x^3-(3n^2-1)x-2\,(n^3-n)$. Vieta's formulas give 
\begin{align*}
& \lambda_{3n-1}(K_{n-1,n,n+1})+\lambda_{3n}(K_{n-1,n,n+1})=-\lambda_1(K_{n-1,n,n+1})=-2n-\delta_n, {\rm and}\\
& \lambda_{3n-1}(K_{n-1,n,n+1})\cdot\lambda_{3n}(K_{n-1,n,n+1})\\
& \qquad=\lambda_1^2(K_{n-1,n,n+1})-(3n^2-1)=(2n+\delta_n)^2-(3n^2-1),
\end{align*}
so we further have
\begin{align}
\nonumber & |\lambda_{3n-1}(K_{n-1,n,n+1})-\lambda_{3n}(K_{n-1,n,n+1})| = \sqrt{(\lambda_{3n-1}(K_{n-1,n,n+1})-\lambda_{3n}(K_{n-1,n,n+1}))^2}=\\
\nonumber & \sqrt{(\lambda_{3n-1}(K_{n-1,n,n+1})+\lambda_{3n}(K_{n-1,n,n+1}))^2-4\lambda_{3n-1}(K_{n-1,n,n+1})\lambda_{3n}(K_{n-1,n,n+1})}=\\
& \sqrt{-3\lambda_1^2(K_{n-1,n,n+1})+12n^2-4}.\label{abs_help}
\end{align}
According to Theorem \ref{cmp_intelace}, it holds
$$n-1\leq-\lambda_{3n-1}(K_{n-1,n,n+1})\leq n\leq -\lambda_{3n}(K_{n-1,n,n+1})\leq n+1,$$
so the spectral distance between the graphs under study is
\begin{equation}\label{dist_formula}
\sigma(K_{n,n,n}, K_{n-1,n,n+1}) = 2n-\lambda_1(K_{n-1,n,n+1})+|\lambda_{3n-1}(K_{n-1,n,n+1})-\lambda_{3n}(K_{n-1,n,n+1})|.
\end{equation}
Now, by implementing (\ref{abs_help}) into (\ref{dist_formula}) we obtain
\begin{equation}\label{final_dist_formula}
\sigma(K_{n,n,n}, K_{n-1,n,n+1})=2n-\lambda_1(K_{n-1,n,n+1})+\sqrt{-3\lambda_1^2(K_{n-1,n,n+1})+12n^2-4}.
\end{equation}
By substituting $\lambda_1(K_{n-1,n,n+1})=2n+\delta_n$, formula (\ref{final_dist_formula}) becomes
\begin{equation*}
\sigma(K_{n,n,n}, K_{n-1,n,n+1})=-\delta_n+\sqrt{-4-12n\delta_n-3\delta_n^2},
\end{equation*}
which, using (\ref{root_delta}), is estimated as follows
\begin{equation*}
\sigma(K_{n,n,n}, K_{n-1,n,n+1})>\frac{4}{9n}+\sqrt{\frac{4}{3}-\frac{3}{n^2}}.
\end{equation*}
For $n\geq4$, it holds
$$\sqrt{\frac{4}{3}}-\sqrt{\frac{4}{3}-\frac{3}{n^{2}}}\leq\frac{3/n^{2}}{\sqrt{4/3}+1}<\frac{1.393}{n^{2}}<\frac{4}{9n}<-\delta_{n},$$
since $\sqrt{\frac{4}{3}-\frac{3}{n^{2}}}\geq 1$ for $n\geq 3$, wherefrom the proof follows.
\end{proof}

We will now expose the main result of the section.

\begin{theorem}\label{refutation}
For every $n\geq3$, it holds
\begin{description}
    \item[(i)] $\sigma^{(2)}(K_{n,n,n},\mathcal{H}_{n})<\sigma^{(2)}(K_{n,n,n},K_{n-1,n,n+1})$,
    \item[(ii)] $\sigma(K_{n,n,n},\mathcal{H}_{n})<\sigma(K_{n,n,n},K_{n-1,n,n+1}).$
\end{description}
\end{theorem}

\begin{proof}
\begin{description}
  \item[(i)] According to Theorem \ref{hn_distances}, we have 
    $$\sigma^{(2)}(K_{n,n,n},\mathcal{H}_{n})=2\left( \Lambda^{2}-\Lambda\lambda_{2}(\mathcal{H}_{n})+\lambda_{2}^{2}(\mathcal{H}_{n})\right),$$
    where by Lemma \ref{spec_bounds_hn}, $\lambda_{2}(\mathcal{H}_{n})\in\left(\frac{1}{2}-\frac{1}{2n},\frac{1}{2}\right)$, while by Remark \ref{remark}, $\Lambda\in\left(0,\frac{1}{2}+\frac{1}{3n}\right)$. Let us denote $f(\Lambda,\lambda_{2}(\mathcal{H}_{n}))=\Lambda^{2}-\Lambda\lambda_{2}(\mathcal{H}_{n})+\lambda_{2}^{2}(\mathcal{H}_{n})$. Analysing the Hessian $\left(
                                                                                                                                                                  \begin{array}{cc}
                                                                                                                                                                    2 & -1 \\
                                                                                                                                                                    -1 & 2 \\
                                                                                                                                                                  \end{array}
                                                                                                                                                                \right)
    $
    of the function $f$, we find that $f$ is strictly convex. The only critical point of $f$ is $(0,0)$, where the global minimum is attained. Therefore, the supremum of the continuous function $f$ on the rectangle $\left(0,\frac{1}{2}+\frac{1}{3n}\right) \times \left(\frac{1}{2}-\frac{1}{2n},\frac{1}{2}\right)$ is equal to its maximum on the closure of the rectangle, that is, it is attained at one of the four corners of the rectangle. If $\Lambda=0$, we have $f(0,\lambda_{2}(\mathcal{H}_{n}))=\lambda_{2}^{2}(\mathcal{H}_{n})<\frac{1}{4},$ while for $\Lambda=\frac{1}{2}+\frac{1}{3n}$, we have $$f\left( \frac{1}{2}+\frac{1}{3n},\frac{1}{2}\right)=\frac{1}{4}+\frac{1}{6n}+\frac{1}{9n^{2}},$$
    and
    $$f\left( \frac{1}{2}+\frac{1}{3n},\frac{1}{2}-\frac{1}{2n}\right) =\frac{1}{4}+\frac{1}{6n}+\frac{1}{9n^{2}}+\frac{1}{2n}\left( \frac{5}{6n}-\frac{1}{2}\right),$$
    where the value in the latter expression does not exceed the value in the former one since $\frac{5}{6n}<\frac{1}{2}.$ Hence,
    $$\sigma^{(2)}(K_{n,n,n},\mathcal{H}_{n})=2f(\Lambda,\lambda_{2}(\mathcal{H}_{n}))\leq\frac{1}{2}+\frac{1}{3n}+\frac{2}{9n^{2}}.$$
    It holds $\frac{1}{2}+\frac{1}{3n}+\frac{2}{9n^{2}}<\frac{2}{3}$, i.e. $\sigma^{(2)}(K_{n,n,n},\mathcal{H}_{n})<\frac{2}{3}$ if and only if $3n^{2}-6n-4>0,$ which is true for $n\geq3.$ Now the proof follows from Lemma \ref{euclid_dist}.
    
  \item[(ii)] According to Theorem \ref{hn_distances} and Remark \ref{remark}, it holds
    $$\sigma(K_{n,n,n},\mathcal{H}_{n})=2\Lambda<1+\frac{2}{3n}<\frac{2\sqrt{3}}{3},$$ for $n\geq 5$. Now the proof follows from Lemma \ref{manhat_dist}. The cases $n=3$ and $n=4$ can be verified by direct computation.
\end{description}
\end{proof}

\begin{remark}
Since $\mathcal{H}_n$ is clearly not isomorphic to either $K_{n,n,n}$ or $K_{n-1,n,n+1}$, according to Theorem \ref{refutation}, both Conjectures \ref{conjecture:triparts:1} and \ref{conjecture:triparts:2} are false for every $n\geq 4$. Notice that the two inequalities of Theorem \ref{refutation} hold for $n = 3$ as well, which is precisely the case which was excluded in \cite{az-arscontemp}.
\end{remark}

The asymptotic behaviour of the four considered spectral distances between the graph $K_{n,n,n}$ and the graph $\mathcal{H}_{n}$, and between $K_{n,n,n}$ and the graph $K_{n-1,n,n+1}$ can be obtained from the formulas for polynomials $q(x)$ and $\Phi(\delta)$ by expanding the corresponding roots.

\begin{theorem}\label{asymptotic}
As $n\rightarrow+\infty,$ it holds
\begin{description}
  \item[(i)] $\sigma(K_{n,n,n},\mathcal{H}_{n})=1+\frac{7}{12n}+O(n^{-2})$, 
  \item[(ii)] $\sigma^{(2)}(K_{n,n,n},\mathcal{H}_{n})=\frac{1}{2}-\frac{1}{12n}+O(n^{-2})$,
  \item[(iii)] $\sigma(K_{n,n,n},K_{n-1,n,n+1})=\frac{2\sqrt{3}}{3}+\frac{4}{9n}+O(n^{-2})$,
  \item[(iv)] $\sigma^{(2)}(K_{n,n,n},K_{n-1,n,n+1})=\frac{2}{3}+\frac{40}{81n^{2}}+O(n^{-4})$.
\end{description}
\end{theorem}

\begin{proof}
We prove (i) and (ii); the expansions (iii) and (iv) are established in \cite[Lemmas 3.7 and 3.8]{az-arscontemp}.

Let us put $t=1/n$ and $z=t\lambda_1(\mathcal{H}_n)$. By (\ref{poly_q}), we have
$$F(z,t)=z^3-z^2-(2-t-t^2)z+t-t^2=0.$$
Since $F(2,0)=0$ and $F_z(2,0)=6\neq0$, where $F_z$ denotes the partial derivative of $F(z,t)$ with respect to $z$, the Implicit Function Theorem gives a unique analytic solution $z=z(t)$ near $t=0$ with $z(0)=2$. By Lemma \ref{spec_bounds_hn}, $t\lambda_1(\mathcal{H}_n)\to2$ as $t\to0$, when $n\to+\infty$.

Let $z(t)=2+at+bt^2+ct^3+O(t^4)$. Substitution into $F(z(t),t)=0$ and comparison of the coefficients of $t$, $t^2$ and $t^3$, respectively, give
$$ 6a+3=0,\qquad 6b+\frac{7}{4}=0,\qquad 6c+\frac{13}{24}=0,$$
where the latter two equations use the values already obtained for the preceding coefficients. Thus,
$$\lambda_1(\mathcal{H}_n)=\frac{z(t)}{t}=2n-\frac{1}{2}-\frac{7}{24n}-\frac{13}{144n^2}+O(n^{-3}).$$
By Theorem \ref{hn_distances}, we have
$$\sigma(K_{n,n,n},\mathcal{H}_n)=2\bigl(2n-\lambda_1(\mathcal{H}_n)\bigr)=1+\frac{7}{12n}+\frac{13}{72n^2}+O(n^{-3}),$$
which proves (i).

For (ii), note first that $\Lambda=2n-\lambda_1(\mathcal{H}_n)=\frac12+\frac{7}{24n}+O(n^{-2})$. Next, $\lambda_2(\mathcal{H}_n)\in(0,\frac12)$ by Lemma \ref{spec_bounds_hn}, and dividing $q(\lambda_2(\mathcal{H}_n))=0$ by $n^2$ gives
$$\frac{\lambda_2^3}{n^2}-\frac{\lambda_2^2}{n}-\Bigl(2-\frac1n-\frac1{n^2}\Bigr)\lambda_2+1-\frac1n=0,$$
where $\lambda_2=\lambda_2(\mathcal{H}_n)$. Hence $\lambda_2=\frac12+O(n^{-1})$, and writing $\lambda_2=\frac12+\frac{\beta}{n}+O(n^{-2})$ and comparing the coefficients of $n^{-1}$ we obtain $-\frac14-2\beta+\frac12-1=0$, i.e.\ $\beta=-\frac38$. Therefore
$$\Lambda^2=\frac14+\frac{7}{24n},\qquad \Lambda\lambda_2=\frac14-\frac{1}{24n},\qquad \lambda_2^2=\frac14-\frac{3}{8n}$$
up to $O(n^{-2})$, and by Theorem \ref{hn_distances},
$$\sigma^{(2)}(K_{n,n,n},\mathcal{H}_n)=2\bigl(\Lambda^2-\Lambda\lambda_2+\lambda_2^2\bigr)=2\Bigl(\frac14-\frac{1}{24n}\Bigr)+O(n^{-2})=\frac12-\frac1{12n}+O(n^{-2}),$$
which proves (ii).

\end{proof}

The last two expansions are in accordance with Lemmas 3.7 and 3.8 of \cite{az-arscontemp}. Notice that the conjectured values of the cospectrality of $K_{n,n,n}$ are therefore wrong by a constant factor, since by Theorem \ref{asymptotic} we have $\mathrm{cs}_{\mathcal{G}}(K_{n,n,n})\leq1+o(1)$ and $\mathrm{cs}_{\mathcal{G}}^{(2)}(K_{n,n,n})\leq\frac{1}{2}+o(1),$ instead of $\frac{2\sqrt{3}}{3}+o(1)$ and $\frac{2}{3}+o(1)$ (based on \cite[Lemmas 3.7 and 3.8]{az-arscontemp}).

\section{Complete multipartite graphs with isolated vertices}

In this section we consider the cospectrality of $K_{n,n,n}$ asymptotically with respect to the Euclidean distance, on the set of all complete multipartite graphs, possibly with isolated vertices added.

\begin{lemma}\label{r_equal_3}
Let $H=tK_{1}\cup K_{p_{1},p_2,\ldots,p_{r}}$ be a graph of order $3n$. If $\sigma^{(2)}(K_{n,n,n},H)<1$, then $r=3$.
\end{lemma}

\begin{proof}
As we have already indicated, the spectrum of $K_{n,n,n}$ consists of the following eigenvalues $2n$, $0$ with the multiplicity $3n-3$, and $-n$ with the multiplicity $2$. 

If $r=1$, $H\cong (t+p_1)K_1$, and since its spectrum consists of entirely zeroes, we have $\sigma^{(2)}(K_{n,n,n},H)=6n^2>1$.

If $r=2$, $H\cong tK_1\cup K_{p_1,p_2}$, and since its spectrum consists of the following eigenvalues $\sqrt{p_1p_2}$, $-\sqrt{p_1p_2}$ and $0$ with the multiplicity $3n-2$, we easily compute $\sigma^{(2)}(K_{n,n,n},H)=(2n-\sqrt{p_1p_2})^2+(-n+\sqrt{p_1p_2})^2+(-n)^2\geq n^2>1$. 

If $r\geq 4$, $H\cong tK_1\cup K_{p_{1},\ldots,p_{r}}$. Let us assume $\sum_{i=1}^{r}p_i=p$ and $1\leq p_1\leq\cdots\leq p_r$. According to Proposition \ref{lambda2=0}, $H$ has exactly one positive eigenvalue, while its $r-1$ negative eigenvalues belong to certain intervals as it is specified by Theorem \ref{cmp_intelace}. Since $K_{n,n,n}$ has exactly three non-zero eigenvalues and all remaining equal to zero, while $H$ has at least $r$ non-zero eigenvalues, it holds $\sigma^{(2)}(K_{n,n,n},H)\geq\lambda_{3n-2}^2(H)=\lambda_{p-2}^2(K_{p_{1},\ldots,p_{r}})\geq p_{r-3}^2\geq 1$.

The proof is completed.  
\end{proof}

\begin{theorem}\label{new_asym}
Let $H=tK_{1}\cup K_{n+a_{1},n+a_{2},n+a_{3}}$ be a graph of order $3n$, where $a_i\in\mathbb{Z}$ are fixed integers and $n+a_i\geq 1$, for $1\leq i\leq 3$, and let $S=\sum_{i=1}^{3} a_{i}(=-t\leq 0)$ and $A=\sum_{i=1}^{3} a_{i}^{2}.$ Then
$$\sigma^{(2)}(K_{n,n,n},H)=\frac{5}{9}S^{2}+\frac{1}{3}A+O(n^{-1}), \quad n\to+\infty.$$
The minimum of the leading term $\frac{5}{9}S^{2}+\frac{1}{3}A$ over all integer triples $(a_{1},a_{2},a_{3})\not=(0,0,0)$ equals $\frac{2}{3},$ and it is attained only for $(a_1,a_2,a_3)=(-1,0,1)$ up to permutation, in which case $S=0$ and $A=2$.
\end{theorem}

\begin{proof}
According to Lemma \ref{helpful}, 
\begin{equation}\label{auxiliary}
\sigma^{(2)}(K_{n,n,n},H)=12n^2+4nS+S^2-A-6n\lambda_1(K_{n+a_{1},n+a_{2},n+a_{3}}).
\end{equation}
Without loss of generality, let us suppose that $a_1\leq a_2\leq a_3$. According to Corollary \ref{cmp_largest_eigen}, $2n+a_1+a_2\leq\lambda_1(K_{n+a_{1},n+a_{2},n+a_{3}})\leq 2n+a_2+a_3$. Therefore, $\lambda_1(K_{n+a_{1},n+a_{2},n+a_{3}})=2n+\delta$, where $\delta\in[a_1+a_2, a_2+a_3]$, that is $\delta=O(1)$. Having in mind the notation introduced in Section 3, it holds 
$$\alpha=3n^2+2\,(a_1+a_2+a_3)n+a_1a_2+a_1a_3+a_2a_3$$ 
and 
$$\beta=2n^3+2n^2(a_1+a_2+a_3)+2(a_1a_2+a_1a_3+a_2a_3)n+2a_1a_2a_3.$$ 
With $P=a_1a_2+a_1a_3+a_2a_3=\frac{S^2-A}{2}$ and $Q=a_1a_2a_3$, we have
$$\phi(x)=x^3-(3n^2+2nS+P)x-(2n^3+2n^2S+2nP+2Q),$$
that is
$$0=\phi(2n+\delta)=\delta^3+6n\delta^2+(9n^2-2nS-P)\delta-2Q-4nP-6n^2S.$$
Using the fact $\delta=O(1)$, after simple calculation, we obtain $\delta=\frac{2}{3}S+\frac{\frac{2}{3}S^{2}-2A}{9n}+O(n^{-2})$. So, (\ref{auxiliary}) becomes:
$$\sigma^{(2)}(K_{n,n,n},H)=4nS+(S^{2}-A)-6n\delta=\frac{5}{9}S^{2}+\frac{1}{3}A+O(n^{-1}), \quad n\to+\infty.$$  
If $S=0$ then $A\geq2$, with equality only for $(a_1,a_2,a_3)=(-1,0,1)$. If $S\neq0$ then $\frac59S^2+\frac13A\geq\frac59+\frac13>\frac23$.
The proof is completed.
\end{proof}

The error term given by Theorem \ref{new_asym} depends on $(a_1,a_2,a_3)$, so the theorem alone does not determine the minimizer for any fixed $n$. Having in mind Lemma \ref{r_equal_3} and the fact $\sigma^{(2)}(K_{n,n,n}, K_{n-1,n,n+1})<\frac{7}{10}$, for $n\geq 4$, we note that $\sigma^{(2)}(K_{n,n,n},tK_{1}\cup K_{a,b,c})>\frac{7}{10}$ for $4\leq n\leq 40$ and every triple $(a,b,c)$, where $1\leq a\leq b\leq c$ and $a+b+c\leq 3n$, di\-ffe\-rent from $(n,n,n)$ and from $(n-1,n,n+1)$.

\section{Concluding remarks}

We have shown that for every $n\geq3$ the graph $\mathcal{H}_{n}$ is closer to $K_{n,n,n}$ than $K_{n-1,n,n+1}$ is, with respect to both the Manhattan and the Euclidean spectral distance, so Conjectures \ref{conjecture:triparts:1} and \ref{conjecture:triparts:2} are false for every $n\geq4$. Moreover, $\mathrm{cs}_{\mathcal{G}}(K_{n,n,n})\leq1+o(1)$ and $\mathrm{cs}_{\mathcal{G}}^{(2)}(K_{n,n,n})\leq\frac{1}{2}+o(1)$, so the conjectured values $\frac{2\sqrt{3}}{3}$ and $\frac{2}{3}$ are wrong by a constant factor. For $n=3$, the graph $\mathcal{H}_{3}$ is the unique graph nearest to $K_{3,3,3}$, with $\sigma(K_{3,3,3},\mathcal{H}_3)\approx1.2188$ and $\sigma^{(2)}(K_{3,3,3},\mathcal{H}_3)\approx0.5748$. Hence $\mathcal{H}_{3}$ is the graph which \cite{az-arscontemp} identified as the nearest one. Let us notice that the graph $F$ drawn in \cite[Figure 1]{az-arscontemp} is obtained from $\mathcal{H}_{3}$ by deleting two non-adjacent edges of its subgraph $K_{3,3}$, and that its distances to $K_{3,3,3}$ are $\sigma(K_{3,3,3},F)\approx4.7732$ and $\sigma^{(2)}(K_{3,3,3},F)\approx3.8058$. So, two edges are apparently missing in that drawing.

\bigskip

{\sloppy\bigskip\noindent\textbf{Acknowledgments.}~~\v{S}krekovski R.
acknowledges the partial support by ARIS program P1-0383 and bilateral
Slovenian-Croatian project BI-HR/25-27-004, while Sedlar J. acknowledges the support by Project
KK.\allowbreak01.\allowbreak 1.\allowbreak1.\allowbreak02.\allowbreak0027
co-financed by the European Regional Development Fund, by Croatian Ministry of
Science, Education and Youth through the bilateral Croatian-Slovenian project
2025-26, and by the NextGeneration EU foundation via IP-UNIST-17 (GEORAZ). }

\bigskip

\noindent\textbf{Declaration of generative AI use.}~~During the preparation of this work the authors used Claude (Anthropic) to assist with editing and preparing the final version of the manuscript, including verification of calculations. The authors reviewed and edited all content and take full responsibility for the paper.

\end{document}